\documentclass[a4paper,12pt]{amsart}
\usepackage[mathscr]{eucal}
\usepackage{amsmath, amssymb, amscd}
\usepackage[english]{babel}
\usepackage{mathabx}
\usepackage{tikz-cd}

\usepackage[colorlinks]{hyperref}
\definecolor{linkblue}{RGB}{1,1,190}
\definecolor{citered}{RGB}{190,1,1}
\hypersetup{
	linkcolor=linkblue,
	urlcolor=linkblue,
	citecolor=citered
}

\theoremstyle{plain}
\newtheorem{theorem}{\bf Theorem}[section]

\newtheorem{lemma}[theorem]{\bf Lemma}
\newtheorem{corollary}[theorem]{\bf Corollary}

\theoremstyle{definition}

\theoremstyle{definition}
\newtheorem{claim}{}[theorem]

\newcommand{\N}{\mathbb N}
\newcommand{\Z}{\mathbb Z}

 \DeclareMathOperator{\ord}{ord}

\renewcommand{\t}{\, | \,}

\newcommand{\be}{\begin{equation}}
\newcommand{\ee}{\end{equation}}

\newcommand{\ber}{\begin{eqnarray}}
\newcommand{\eer}{\end{eqnarray}}

\begin{document}
\title[Separating Noether number]{On the separating Noether number of finite abelian groups}

\author{Barna Schefler and Kevin Zhao and Qinghai Zhong}

\address{School of Mathematics and Statistics, Nanning Normal University, Nanning 530100, China, and Center for Applied Mathematics of  Guangxi, Nanning Normal University, Nanning 530100, China}
\email{zhkw-hebei@163.com}

\address{E\"otv\"os Lor\'and University, 
P\'azm\'any P\'eter s\'et\'any 1/C, 1117 Budapest, Hungary} 
\email{scheflerbarna@yahoo.com}

\address{
	Institute of Mathematics and Scientific Computing \\
	Heinrichstra{\ss}e 36\\
	8010 Graz, Austria}
\email{qinghai.zhong@uni-graz.at}
\urladdr{https://imsc.uni-graz.at/zhong/}

\thanks{This research was funded in part by National Science Foundation
	of China Grant \#12301425 and the Austrian Science Fund (FWF) [grant DOI:10.55776/P36852]}


\maketitle

\begin{abstract} The separating Noether number  $\beta_{\mathrm{sep}}(G)$ of a finite group $G$ is the minimal positive integer $d$ such that for every finite $G$-module $V$ there is a separating set consisting of invariant polynomials of degree at most $d$. In this paper we use methods from additive combinatorics to investigate the separating Noether number for finite abelian groups. Among others, we obtain the exact value of  $\beta_{\mathrm{sep}}(G)$, provided that  $G$ is either a $p$-group or has rank $2$, $3$ or $5$.
\end{abstract}

\section{Introduction}

 Let $V$ be a finite-dimensional vector space over  the field of complex numbers $\mathbb{C}$ and let $G$ be a finite group. Suppose that there is a $G$-module structure on $V$. Let $x_1,...,x_k$ be a basis of the dual space $V^*$ of $V$. Then the coordinate ring $\mathbb{C}[V]$ of $V$ can be identified with the polynomial
algebra $\mathbb{C}[x_1,...,x_k]$. The $G$-module structure on $V$ induces the following $G$-action on  $\mathbb{C}[V]$:
\begin{center}
    for $g\in G$ and $f\in\mathbb{C}[V]$, $g\cdot f\in\mathbb{C}[V]$ where $g\cdot f(v)=f(g^{-1}v)$. 
\end{center}
Consider the invariant subalgebra
\[\mathbb{C}[V]^G:=\{f\in\mathbb{C}[V]:g\cdot f=f\text{ for each } g\in G\}.\]
The $G$-action preserves the standard grading of $\mathbb{C}[x_1,...,x_k]$, whence the invariant subalgebra is generated by homogeneous polynomials. A theorem of Noether states that it is finitely generated by homogeneous polynomials of degree $\leq |G|$ (see \cite{N16}). Motivated by this result, the Noether number $\beta(G)$ of a finite group $G$ is defined as the supremum of the maximal degree of an element in a minimal homogeneous generating system of $\mathbb{C}[V]^G$, where $V$ ranges over all finite dimensional $G$-modules. For finite non-abelian groups, the precise value of the Noether number $\beta (G)$ was known only for the dihedral groups and very few small groups (such as $A_4)$ (see \cite{Cz19,Cz-Do14,Cz-Do-Sz18}).

In 1991, B. Schmid \cite{Sc91a} observed that computing $\beta(G)$ for a finite abelian group $G$ can be connected to  zero-sum theory (for a survey on this interplay, we refer to \cite{Cz-Do-Ge16}). Indeed, we have $\beta(G)=\mathsf D(G)$, where the Davenport constant $\mathsf D(G)$ is the maximal length of a minimal zero-sum sequence over $G$. It is well-known that for every finite abelian group $G$ with $|G|\ge 2$, we have
\[G\cong C_{n_1}\oplus \ldots\oplus C_{n_r}\text{ with }1<n_1\t \ldots \t n_r\,,
\]
and 
\[\mathsf D(G)\ge \mathsf D^*(G):=1+\sum_{i=1}^r(n_1-1).\] 
Equality holds  for $p$-groups, rank at most $2$ groups  and some other special groups (see Lemma \ref{D}). It is still  open for groups of rank $3$ and groups of the form $C_n^r$. However, there are infinitely many groups with rank larger than $3$ such that the strict inequality holds (\cite{C20,B69, GS92, M92}). For more on the Davenport constant, one can see \cite{GG06,Ge09a,Ge-HK06a,[G], Gi18, GS19, GS20,PS11,CS14,S11}.

A branch of invariant theory studies separating systems of polynomial invariants. A subset $S\subset{\mathbb{C}}[V]^G$ is called a separating system if for any  two elements $v_1,v_2\in V$ that can be separated by a polynomial invariant, there exists an $f\in S$ such that $f(v_1)\neq f(v_2)$. For finite groups we have the following characterization:
\begin{center}
$S\subset{\mathbb{C}}[V]^G$ is a separating set if and only if for any $v, w\in V$ with $Gv \neq Gw$ there exists $f\in S$ such that $f(v) \neq f(w)$.
\end{center}
 The \textbf{separating Noether number} $\beta_{\mathrm{sep}}(G)$ of a finite group $G$ was introduced in  \cite{Ko-Kr10a} as the minimal positive integer $d$ such that for every finite $G$-module $V$ there is a separating set consisting of invariant polynomials of degree at most $d$. By \cite[Theorem 3.12.1]{derksen-kemper}, invariant polynomials of degree at most $|G|$ form a separating set, whence
 we have $\beta_{\mathrm{sep}}(G)\le |G|$.
For finite abelian groups, the exact value of $\beta_{\mathrm{sep}}(G)$ is known for groups with $\mathsf {r}(G)\le 2$ or $G\cong C_n^r$ (see \cite{Sc24a,Sc25a}). Furthermore, there is research on separating invariants by considering  $G$-module $V$  over  fields of  finite characteristic (we refer to \cite[Section 5]{Ko-Kr10a} for interesting readers).

In this paper we continue to study the exact value of the separating Noether number for finite abelian groups. The following is our main theorem.

\begin{theorem}\label{main}
	Let $G=C_{n_1}\oplus \ldots\oplus C_{n_r}$ with $1<n_1\t \ldots \t n_r$ and $r\ge 2$. Suppose $\mathsf D(n_sG)=\mathsf D^*(n_sG)$, where $s=\lfloor\frac{r+1}{2}\rfloor$.
	Then
	\[\begin{cases}
		\beta_{\mathrm{sep}}(G)= n_s+n_{s+1}+\ldots+n_r,&\mbox{ if }r \mbox{ is odd} \\
		\beta_{\mathrm{sep}}(G)\le \frac{n_{s}}{p}+n_{s+1}+\ldots+n_r,&\mbox{ if }r\mbox{ is even},\\
	\end{cases}\]
	where $p$ is the minimal prime divisor of $n_s$.
\end{theorem}

So far, there is known no group $G$ of odd rank, where equality does not hold.
Note that $\mathsf D(n_sG)=\mathsf D^*(n_sG)$ holds true when $n_sG$ is a $p$-group, or has rank at most $2$,  or are of some other special cases (see Lemma \ref{D}). 
In particular, by applying Theorem \ref{main},  we have the following corollaries.

\begin{corollary}\label{co1}
	Let $G=C_{n_1}\oplus \ldots\oplus C_{n_r}$ with $1<n_1\t \ldots \t n_r$ be a $p$-group.
	Then
	\[\begin{cases}
		\beta_{\mathrm{sep}}(G)= n_s+n_{s+1}+\ldots+n_r,&\mbox{ if }r \mbox{ is odd} \\
		\beta_{\mathrm{sep}}(G)= \frac{n_{s}}{p}+n_{s+1}+\ldots+n_r,&\mbox{ if }r\mbox{ is even}\,,\\
	\end{cases}\]
where $s=\lfloor\frac{r+1}{2}\rfloor$.
\end{corollary}

\begin{corollary}\label{co2}
	Let $G=C_{n_1}\oplus \ldots\oplus C_{n_r}$ with $1<n_1\t \ldots \t n_r$ and $r\in \{2,3,5\}$, and let $s=\lfloor\frac{r+1}{2}\rfloor$.
	Then
\[\begin{cases}
\beta_{\mathrm{sep}}(G)= \frac{n_1}{p}+n_2,&\mbox{ if }r=2\,,\\
\beta_{\mathrm{sep}}(G)= n_2+n_3,&\mbox{ if }r =3\,,\\
\beta_{\mathrm{sep}}(G)= n_3+n_4+n_5,&\mbox{ if }r =5\,. \\
\end{cases}\]
\end{corollary}

Note that the above corollary for finite abelian groups of rank $2$ was obtained in \cite{Sc24a} by using a different method. If $r=4$, then the upper bound given by Theorem \ref{main} does not meet the known lower bound (see Lemma \ref{lowerbound}) , whence this case is still open.

	\medskip
In Section 2, we fix our notation and gather the required tools. 
Our main results will be proved in Section 3.

\section{Preliminaries}\label{sec:separating-Davenport}

 Let  $\mathbb N$ be  the set of positive integers, let $\N_0=\N\cup\{0\}$, and let $\Z$ be the set of integers. For integers $a, b$, we denote by
$[a, b ] = \{ x \in \mathbb Z \mid a \le x \le b\}$  the discrete, finite interval between $a$ and $b$.
Let $G$ be a finite abelian group and let $C_n$ be an additive cyclic group of order $n$. If $|G|\ge 2$, then $G \cong C_{n_1} \oplus \ldots \oplus C_{n_r}$ with $1<n_1\t n_2\t \ldots \t n_r$, where  $r = \mathsf r (G)$ is the rank of $G$ and $n_r=\exp(G)$ is the exponent of $G$.

Let $G_0 \subset G$ be a nonempty subset. In Additive Combinatorics, a { sequence} (over $G_0$) means a finite unordered sequence of terms from $G_0$ where repetition is allowed, and (as usual) we consider sequences as elements of the free abelian monoid $\mathcal F (G_0)$ with basis $G_0$. Let
\[
S = g_1 \ldots g_{\ell} = \prod_{g \in G_0} g^{\mathsf v_g (S)} \in \mathcal F (G_0)
\]
be a sequence over $G_0$, where $\mathsf v_g(S)$ is the multiplicity of $g$ in $S$. We call
\[
\begin{aligned}
|S|  &= \ell = \sum_{g \in G} \mathsf v_g (S) \in \mathbb N_0 \
\text{the \ {\it length} \ of \ $S$} \,,  \\
\sigma (S) & = \sum_{i = 1}^{\ell} g_i \ \text{the \ {\it sum} \ of \
	$S$} \,, \\
\ \ \text{ and }\ \ \  \Sigma (S) &= \Big\{ \sum_{i \in I} g_i
\mid \emptyset \ne I \subset [1,\ell] \Big\} \ \text{ the \ {\it set of
		subsequence sums} \ of \ $S$} \,.
\end{aligned}
\]
The sequence $S$ is said to be
\begin{itemize}
	\item {\it zero-sum free} \ if \ $0 \notin \Sigma (S)$,
	
	\item a {\it zero-sum sequence} \ if \ $\sigma (S) = 0$,
	
	\item a {\it minimal zero-sum sequence} (or an {\it atom}) \ if it is a nontrivial zero-sum
	sequence and every proper  subsequence is zero-sum free.
\end{itemize}
The set of zero-sum sequences $\mathcal B (G_0) = \{S \in \mathcal F (G_0) \colon \sigma (S)=0\} \subset \mathcal F (G_0)$ is a submonoid. We denote by $\mathcal A(G_0)$ the set of all atoms over $G_0$ and by
\[
\mathsf D (G_0) = \max \{ |S| \colon S \in \mathcal A (G_0) \} \in \N
\]
the { Davenport constant} of $G_0$. If $A$ is an atom over $G_0$,  then $g^{-1}A$ is zero-sum free for any $g\t A$. If $S$ is zero-sum free over $G$, then $|S|\le \mathsf D(G)-1$. 

We  collect known results concerning  the Davenport constant.
\begin{lemma}\label{D}
	Let $G$ be a finite abelian group. Then $\mathsf D(G)=\mathsf D^*(G)$, provided that one of the following holds.
	\begin{enumerate}
		\item[(a)] $r(G)\leq 2$; (\cite[Theorem 4.2.10]{Ge09a} or \cite[Theorem 5.8.3]{Ge-HK06a})
		\item[(b)] $G$ is a finite abelian $p$-group; (\cite[Theorem 2.2.6]{Ge09a} or \cite[Theorem 5.5.9]{Ge-HK06a})
		\item[(c)] $G= K \oplus C_{km}$, where $K$ is a  $p$-group with $\mathsf{D}(K)\leq m$ and $m$ is a power of $p$; (\cite[Corollary 4.2.13]{Ge09a})
		\item[(d)] $G = C_2 \oplus C_{2m} \oplus C_{2n}$ with $m|n$; (\cite{B69,GGG10})
		\item[(e)] $G = C_3 \oplus C_{6m} \oplus C_{6n}$ with $m|n$; (\cite{B69,GGG10})
		\item[(f)] $G = C_{2p^a} \oplus C_{2p^b} \oplus C_{2p^c}$ with $a\leq b \leq c$ being nonnegative integers, and $p$ being a prime; (\cite{B69})
		\item[(g)] $G = C^3_2 \oplus C_{2n}$; (\cite{B69})
		\item[(h)] $G = C^4_2 \oplus C_{2n}$ with $n\geq 70$ and $n$ even; (\cite[Theorem 5.8]{CS14})
		\item[(i)] $G = C^5_2 \oplus C_{2n}$ with $n\geq 149$ and $n$ even. (\cite[Theorem 2]{Z23})
	\end{enumerate}
\end{lemma}

Let $H\subset \mathcal F(G_0)$ be a submonoid. Then $H$ is  commutative and cancellative. We denote by $\mathsf q(H)$ its quotient group. For every $\ell\in \N_0$, we denote by $\mathcal B(G_0)_{\ell}$ the submonoid generated by the set $\{A\in \mathcal A(G_0)\colon |A|\le \ell\}$ and we call an atom $A\in \mathcal A(G_0)$ is a \textbf{separating atom}  over $G_0$ if $$A\not\in \mathsf q(\mathcal B(G_0)_{|A|-1})\,,$$
that is there do not exist atoms $U_1,\ldots, U_k, V_1,\ldots, V_l\in \mathcal A(G_0)$ with $|U_i|,|V_j|\le |A|-1$ for all $i,j$ such that $AU_1\cdot\ldots\cdot U_k=V_1\cdot\ldots\cdot V_{l}$.
Set $\mathcal A_{\mathrm{sep}}(G_0)$ be the set of all separating atoms over $G_0$. Then
\begin{equation}\label{generatingsubset}
 \mathcal A_{\mathrm{sep}}(G_0) \text{ is a generating subset of  } \mathsf q(\mathcal B(G_0))\,.
\end{equation}

\begin{lemma}\label{separatingatomlength}
	Let $G$ be a finite abelian group and let $G_0\subset G$ be a nonempty subset.
 If $A$ is a separating atom over $G_0$, then $|A|\le \mathsf D^*(G)$.
\end{lemma}

\begin{proof}
	By \cite[Proposition 3.9.(i)]{Do17a}, we have $\{S\in \mathcal B(G_0)\colon |S|\le \mathsf D^*(G)\}$ is a generating subset of $\mathsf q(\mathcal B(G_0))$, whence
	\[
	\mathsf q(\mathcal B(G_0)_{\ell})=\mathsf q(\mathcal B(G_0))\text{ for all }\ell\ge \mathsf D^*(G)\,.
	\]
	Let $A$ be a separating atom over $G_0$. Assume to the contrary that $|A|-1\ge \mathsf D^*(G)$. Then the above equation implies that $A\in \mathsf q(\mathcal B(G_0))=\mathsf q(\mathcal B(G_0)_{|A|-1})$, a contradiction to the definition of separating atoms.
\end{proof}

\begin{lemma}\label{atom}
	Let $G$ be a finite abelian group. Then
 $\beta_{\mathrm{sep}}(G)$ is the maximal length of a separating atom of $\mathcal{B}(G_0)$, where $G_0\subset G$ ranges over all subsets of size $\leq \mathsf r(G)+1$.
\end{lemma}
\begin{proof}
	See \cite[Corollary 2.6.]{Do17a}.
\end{proof}

\begin{lemma}\label{lowerbound}
		Let $G=C_{n_1}\oplus \ldots\oplus C_{n_r}$ with $1<n_1\t \ldots \t n_r$.
Then
\[\beta_{\mathrm{sep}}(G)\ge \begin{cases}
	n_s+n_{s+1}+\ldots+n_r,&\mbox{ if }r \mbox{ is odd}, \\
	\frac{n_{s}}{p_1}+n_{s+1}+\ldots+n_r,&\mbox{ if }r\mbox{ is even},\\
\end{cases}\]
where $s=\lfloor\frac{r+1}{2}\rfloor$ and $p_1$ is the minimal prime divisor of $n_1$. In particular, we have $$\beta_{\mathrm{sep}}(G)>n_{s+1}+\ldots+n_r\,.$$
\end{lemma}
\begin{proof}
	See \cite[Lemmas 5.2 and 5.5]{Sc25a}.
\end{proof}

\begin{lemma}\label{upperbound}
	Let $G=C_{n_1}\oplus \ldots\oplus C_{n_r}$ with $1<n_1\t \ldots \t n_r$.
	Then \[\beta_{\mathrm{sep}}(G)\le
	\begin{cases}
		sn_r,&\mbox{ if }r \mbox{ is odd}\,, \\
		\frac{n_r}{p}+sn_r,&\mbox{ if }r\mbox{ is even}\,,\\
	\end{cases}\]
 where $s=\lfloor\frac{r+1}{2}\rfloor$ and $p$ is the smallest prime divisor of $n_r$.
\end{lemma}

\begin{proof}
	Let $G^*=C_{n_r}^r$. Then $G$ is isormorphic to a subgroup of $G^*$, whence $\beta_{\mathrm{sep}}(G^*)\ge \beta_{\mathrm{sep}}(G)$ by \cite[Theorem B]{Ko-Kr10a}. It suffices to show that the upper bound assertion holds for $G^*$ and this follows from \cite[Theorem 1.2]{Sc25a}.
\end{proof}


\section{Proof of the main theorem and its corollaries}

\begin{proof}[Proof of Theorem \ref{main}]
	If $n_sG$ is trivial, then $n_s=\exp(G)=n_r$ and hence $n_s=n_{s+1}=\ldots=n_r$. The assertion now follows from Lemmas \ref{upperbound} and \ref{lowerbound}.
	
	Suppose $n_sG$ is not trivial. Then the assumption $\mathsf D(n_sG)=\mathsf D^*(n_sG)$ implies that
	\begin{equation}\label{eq*}
		\mathsf D(n_sG)=\mathsf D^*(n_sG)=1+\sum_{j=s+1}^r(\frac{n_j}{n_s}-1)=1+\sum_{j=s}^r(\frac{n_j}{n_s}-1)\,.
	\end{equation}
	 By Lemma \ref{atom}, we may choose a subset $G_0\subset G$  with $|G_0|\le r+1$ such that there exists a separating atom $A$ over $G_0$ with $|A|=\beta_{\mathrm{sep}}(G)$. Set $G_0=\{g_1,g_2,\ldots,g_{|G_0|}\}$ and $G_1=\{n_sg\colon g\in G_0\}$. We define a map 
	 \[
	 \varphi\colon \{S\in \mathcal F(G_0)\colon n_s\t \mathsf v_g(S) \text{ for each }g\in G_0\}\rightarrow \mathcal F(G_1) 
	 \] by $\varphi(\prod_{g\in G_0}g^{n_sy_g})=\prod_{g\in G_0}(n_sg)^{y_g}$, where $y_g, g\in G_0$ are non-negative integers. For every sequence $T$ over $G_1$, $\varphi^{-1}(T)$ is the set of all sequences $S$ with $n_s\t \mathsf v_g(S)$ for each $g\in G_0$ such that $\varphi(S)=T$.

	\begin{claim}\label{cl1}
		Let $S\in \mathcal B(G_0)$ with $n_s\t \mathsf v_g(S)$ for each $g\in G_0$. Then $S\in \mathsf q(\mathcal B(G_0)_{|A|-1})$.
	\end{claim}
	\begin{proof}[Proof of \ref{cl1}]
	Since $\varphi(S)$ is a zero-sum sequence over $G_1$, it follows from  \eqref{generatingsubset} that we may factor $\varphi(S)=U_1\cdot\ldots\cdot U_{\ell}\cdot U_{\ell+1}^{-1}\cdot\ldots\cdot U_k^{-1}$, where $U_1,\ldots, U_{k}\in \mathcal A_{\mathrm{sep}}(G_1)$. Therefore by choosing suitable subsequences $\varphi^*(U_i)$ from the set $\varphi^{-1}(U_i)$ for each $i\in [1,k]$, we have that 
	\[
	S=\varphi^*(U_1)\cdot\ldots\cdot \varphi^*(U_{\ell})\cdot \varphi^*(U_{\ell+1})^{-1}\cdot\ldots\cdot \varphi^*(U_k)^{-1}.
	\]
	In view of  Lemmas \ref{separatingatomlength} and \ref{lowerbound}, for every $i\in [1,k]$, we have
 $$|\varphi^*(U_i)|=n_s|U_i|\le n_s\mathsf D^*(n_sG)= \sum_{j=s+1}^rn_j-(r-s-1)n_s\le \beta_{\mathrm{sep}}(G)-1=|A|-1\,.$$ Now the assertion follows.
	\qedhere(\ref{cl1})	
	\end{proof}

	 We may assume that $A=\prod_{i=1}^{|G_0|}g_i^{m_i}$, where $m_i\in \N_0$ for every $i\in [1,|G_0|]$. Let $p$ be the minimal prime divisor of $n_s$ and let $l\in [1, p-1]$.
	For every $i\in [1, |G_0|]$, there exist $k_i^{(l)}\in \N_0$ and $x_i^{(l)}\in [0, n_s-1]$ such that $lm_i=k_i^{(l)}n_s+x_i^{(l)}$.

	If $x_i^{(1)}=0$ for every $i\in [1,|G_0|]$, then $A\in \mathcal B(G_0)$ with $n_s\t \mathsf v_g(S)$ for every $g\in G_0$. It follows from \ref{cl1} that $A\in \mathsf q(\mathcal B(G_0)_{|A|-1})$, a contradiction to that $A$ is a separating atom.
	 Thus there exists some $i_0\in [1, |G_0|]$ such that  $x_{i_0}^{(1)}\neq 0$. Since $\gcd(l,n_s)=1$ and $x_{i_0}^{(l)}\equiv lx_{i_0}^{(1)}\pmod {n_s}$, we have
\[
 x_{i_0}^{(l)}\neq 0\,.
\]

Let $A^{(l)}=\prod_{i=1}^{|G_0|}g_i^{k_i^{(l)}n_s}$. Then $\varphi(A^{(l)})$ is a sequence over $n_sG$, and hence we can write $\varphi(A^{(l)})=X_0^{(l)}\cdot X_1^{(l)}$, where $X_1^{(l)}$ is a zero-sum sequence over $n_sG$ (note that $X_1^{(l)}$ is trivial when $l=1$) and $X_0^{(l)}$ is a zero-sum free sequence over $n_sG$. We may choose suitable $\varphi^*(X_0^{(l)})\in \varphi^{-1}(X_0^{(l)})$ and $\varphi^*(X_1^{(l)})\in \varphi^{-1}(X_1^{(l)})$ such that $A^{(l)}=\varphi^*(X_0^{(l)})\varphi^*(X_1^{(l)})$.
It follows from \eqref{eq*} that
\begin{equation}\label{X}
	|\varphi^*(X_0^{(l)})|=n_s|X_0^{(l)}|\le n_s(\mathsf D(n_sG)-1)=\sum_{j=s+1}^rn_j-(r-s)n_s\,.
\end{equation}
 Set $W^{(l)}=\varphi^*(X_0^{(l)})\prod_{i=1}^{|G_0|}g_i^{x_i^{(l)}}$. Then $W^{(l)}=\varphi^*(X_1^{(l)})^{-1}A^l$ is zero-sum.
Let $l'\in [1,n_s-1]$ such that $ll'\equiv 1\pmod {n_s}$. Then there exists  $h^{(l)}\in \N_0$ such that $ll'=1+h^{(l)}n_s$.
It follows that
\[
A=A^{ll'-h^{(l)}n_s}=(A^{h^{(l)}n_s})^{-1}(A^l)^{l'}=(A^{h^{(l)}n_s})^{-1}(\varphi^*(X_1^{(l)})\cdot W^{(l)})^{l'}\,.
\]
By \ref{cl1}, we have  $A^{h^{(l)}n_s}\in \mathsf q(\mathcal B(G_0)_{|A|-1})$
and $\varphi^*(X_1^{(l)})\in \mathsf q(\mathcal B(G_0)_{|A|-1})$.
Since $A$ is a separating atom, we obtain
\begin{equation}\label{W}
	|W^{(l)}|\ge |A|\,.
\end{equation}

Note that $m_i\le \ord(g_i)-1$. We have $k_i^{(l)}\le l\ord(g_i)-1$ and hence
\[
\sigma\left(\prod_{i=1}^{|G_0|}g_i^{(l\ord(g_i)-k_i^{(l)})n_s-x_i^{(l)}}\right)=\sigma\left(\prod_{i=1}^{|G_0|}g_i^{l\ord(g_i)n_s}\right)-\sigma(A^l)=0\,,
\]
which means that there exist $t_1,\ldots, t_{|G_0|}\in \N$ such that $\sigma\left(\prod_{i=1}^{|G_0|}g_i^{t_in_s-x_i^{(l)}}\right)=0$.
We may choose a tuple $(t_1^{(l)},\ldots, t_{|G_0|}^{(l)})\in \N^{|G_0|}$ with $\sum_{j=1}^{|G_0|}t_j^{(l)}$ minimal such that $\sigma\left(\prod_{i=1}^{|G_0|}g_i^{t_i^{(l)}n_s-x_i^{(l)}}\right)=0$.
Set $$V^{(l)}=\prod_{i=1}^{|G_0|}g_i^{t_i^{(l)}n_s-x_i^{(l)}} \quad \text{ and }\quad   Y^{(l)}=\prod_{i=1}^{|G_0|}g_i^{(t_i^{(l)}-1)n_s}\,.$$
Assume to the contrary that $\varphi(Y^{(l)})$ is not zero-sum free. Then $\varphi(Y^{(l)})$ has a nontrivial zero-sum subsequence, say  $\prod_{i=1}^{|G_0|}(n_sg_i)^{t_i'}$, where $t_i'\in [0, t_i^{(l)}-1]$ for every $i\in [1,|G_0|]$. Therefore $\sum_{i=1}^{|G_0|}t_i'>0$, $t_i^{(l)}-t_i'\ge 1$ for every $i\in [1,|G_0|]$, and
\[
\sigma\left(\prod_{i=1}^{|G_0|}g_i^{(t_i^{(l)}-t_i')n_s-x_i^{(l)}}\right)=\sigma(V^{(l)})-\sigma\left(\prod_{i=1}^{|G_0|}(n_sg_i)^{t_i'}\right)=0\,.
\]
We obtain a contradiction to the minimality of $\sum_{i=1}^{|G_0|}t_i^{(l)}$. Thus $\varphi(Y^{(l)})$ is zero-sum free and it follows from \eqref{eq*} that
\begin{equation}\label{eq2}
	n_s(\sum_{i=1}^{|G_0|}(t_i^{(l)}-1))=|Y^{(l)}|=n_s|\varphi(Y^{(l)})|\le n_s(\mathsf D(n_sG)-1)=\sum_{j=s+1}^{r}n_j-(r-s)n_s\,.
\end{equation}
Note that
\[
A=A^{ll'-h^{(l)}n_s}=(A^{h^{(l)}n_s})^{-1}(A^l)^{l'}=(A^{h^{(l)}n_s})^{-1}(V^{(l)})^{-l'}(A^lV^{(l)})^{l'}\,.
\]
By \ref{cl1}, we have  $A^{h^{(l)}n_s}\in \mathsf q(\mathcal B(G_0)_{|A|-1})$
and $A^lV^{(l)}=\prod_{i=1}^{|G_0|}g_i^{(k_i^{(l)}+t_i^{(l)})n_s}\in \mathsf q(\mathcal B(G_0)_{|A|-1})$.
Since $A$ is a separating atom, we obtain
\begin{equation}\label{V}
	|V^{(l)}|\ge |A|\,.
\end{equation}
Then \eqref{W}, \eqref{V}, \eqref{X}, and \eqref{eq2} imply
\begin{align*}
2|A|\le & |W^{(l)}|+|V^{(l)}|=|\varphi^*(X_0^{(l)})|+n_s\sum_{i=1}^{|G_0|}t_i^{(l)}\\
	&\le \sum_{j=s+1}^rn_j-(r-s)n_s+\sum_{j=s+1}^rn_j+(s+1)n_s\\
	&=n_s(2s+1-r)+2\sum_{j=s+1}^rn_j
\end{align*}

Suppose $r$ is odd. Then $r=2s-1$ and hence $\beta_{\mathrm{sep}}(G)=|A|\le \sum_{j=s}^rn_j$.
 Now the assertion follows from Lemma \ref{lowerbound}.

 Suppose $r$ is even. Then $r=2s$ and hence
 \begin{equation}\label{W+V}
 	2|A|\le |W^{(l)}| + |V^{(l)}| \le n_s+2\sum_{j=s+1}^rn_j\,.
 \end{equation}
 Furthermore, if $|X_0^{(l)}|\le \mathsf D(n_sG)-2$, then  $|\varphi^*(X_0^{(l)})|\le \sum_{j=s+1}^rn_j-(r-s+1)n_s$ and hence
 $2\beta_{\mathrm{sep}}(G)=2|A|\le 2\sum_{j=s+1}^rn_j<2\beta_{\mathrm{sep}}(G)$ by Lemma \ref{lowerbound}, a contradiction.
 Thus $|X_0^{(l)}|=\mathsf D(n_sG)-1$ and hence
 \begin{equation}\label{eqX0}
 	|\varphi^*(X_0^{(l)})|=\sum_{j=s+1}^rn_j-(r-s)n_s=\sum_{j=s+1}^rn_j-sn_s\,.
 \end{equation}

 Assume to the contrary that $\beta_{\mathrm{sep}}(G)=|A|>\frac{n_s}{p}+n_{s+1}+\ldots+n_r$. Then
 $V^{(l)}\ge |A|>\frac{n_s}{p}+n_{s+1}+\ldots+n_r$ and hence  (in view of \eqref{W+V})
 \begin{align*}
 &\frac{n_s}{p}+n_{s+1}+\ldots+n_r< |A|\le |W^{(l)}|\\
 =&|\varphi^*(X_0^{(l)})|+\sum_{i=1}^{|G_0|}x_i^{(l)}
 = (|W^{(l)}| + |V^{(l)}|)-|V^{(l)}|
 < n_{s}+\ldots+n_r-\frac{n_s}{p}\,.
 \end{align*}
 It follows from \eqref{eqX0} that
 \[
 \frac{n_s}{p}<\sum_{i=1}^{|G_0|}x_i^{(l)}-sn_s<n_s-\frac{n_s}{p}\quad  \text{ holds for every }l\in [1,p-1]\,,
 \]
whence $p\ge 3$.
Set $c^{(l)}=\sum_{i=1}^{|G_0|}x_i^{(l)}-sn_s$ for every $l\in [1,p-1]$. Let $l',l^*\in [1,p-1]$ be distinct such that $c^{(l')}\ge c^{(l^*)}$. Then  $|l'-l^*|\in [1,p-1]$ and hence
 \begin{align*}
n_s-2\frac{n_s}{p}>c^{(l')}-c^{(l^*)}=\sum_{i=1}^{|G_0|}x_i^{(l')}-\sum_{i=1}^{|G_0|}x_i^{(l^*)}
\begin{cases}
\equiv c^{(l'-l^*)} \pmod {n_s}, &\text{ if } l'>l^*,\\
\equiv n_s-c^{(l^*-l')} \pmod {n_s}, &\text{ if } l^*>l'.
\end{cases}
 \end{align*}
 In view of $\frac{n_s}{p}<c^{(|l'-l^*|)}<n_s-\frac{n_s}{p}$, we have
 \[
   c^{(l')}-c^{(l^*)}>\frac{n_s}{p}\,.
 \]
 Suppose $l_1,l_2,\ldots,l_{p-1}\in [1, p-1]$ are distinct such that
 $c^{(l_1)}>c^{(l_2)}>\ldots>c^{(l_{p-1})}$.
 It follows that
 \[
  n_s-2\frac{n_s}{p}> c^{(l_1)}-c^{(l_{p-1})}=\sum_{j=1}^{p-2}\left(c^{(l_{j})}-c^{(l_{j+1})}\right)> (p-2)\frac{n_s}{p}=n_s-2\frac{n_s}{p}\,,
 \]
 a contradiction.
\end{proof}

\begin{proof}[Proof of Corollary \ref{co1}]	
	Since $G$ is a $p$-group,   $n_sG$ is  a $p$-group and hence Lemma \ref{D}.(b) implies that the assumption of Theorem \ref{main} holds for $n_sG$.
	Since $p\t n_1$ and $p$ is the least prime divisor dividing $n_{s}$, we can now apply Theorem \ref{main} and Lemma \ref{lowerbound} to obtain the assertion.
\end{proof}

\begin{proof}[Proof of Corollary \ref{co2}]	
	Since $r\in \{2,3,5\}$, we have $\mathsf r(n_sG)\le 2$ and Lemma \ref{D}.(a) implies that the assumption of Theorem \ref{main} holds for $n_sG$. If $r$ is odd,  then the assertion follows by applying Theorem \ref{main}. If $r=2$, then the assertion follows from Theorem \ref{main} and Lemma \ref{lowerbound}
\end{proof}

\providecommand{\bysame}{\leavevmode\hbox to3em{\hrulefill}\thinspace}
\providecommand{\MR}{\relax\ifhmode\unskip\space\fi MR }
\providecommand{\MRhref}[2]{%
	\href{http://www.ams.org/mathscinet-getitem?mr=#1}{#2}
}
\providecommand{\href}[2]{#2}


\begin{thebibliography}{1}
	
\bibitem{Cz19}	K. Cziszter,\emph{The Noether number of p-groups}, J. Algebra Appl. \textbf{18}(2019), 1950066.

\bibitem{Cz-Do14}
K. Cziszter and M. Domokos,
\emph{The Noether number for the groups with a cyclic subgroup of index two},
J. Algebra,
\textbf{399}(2014), 546-560.

\bibitem{Cz-Do-Ge16}
K.~Cziszter, M.~Domokos, and A.~Geroldinger, \emph{The interplay of invariant theory with multiplicative ideal theory and with arithmetic combinatorics}, Multiplicative {I}deal {T}heory and {F}actorization {T}heory, Springer, 2016, pp.~43 -- 95.

\bibitem{Cz-Do-Sz18}
K. Cziszter, M. Domokos, and I. Sz{\" o}ll{\H o}si,
\emph{The Noether numbers and the Davenport constants of the groups of order less than 32},
J. Algebra,
\textbf{510}(2018), 513-541.
	
	\bibitem{derksen-kemper}
	H.~Derksen and G.~Kemper, \emph{Computational invariant theory}, Encyclopedia
	of Mathematical Sciences, vol. 130.
	
	\bibitem{Do17a}
	M.~Domokos, \emph{Degree bound for separating invariants of abelian groups},
	Proc. Amer. Math. Soc. \textbf{145} (2017), 3695 -- 3708.




	
\bibitem{B69} P. van Emde Boas, {\it A combinatorial problem on finite abelian groups II}, Reports ZW1969C007, Mathematical Centre, Amsterdam, 1969.

\bibitem{GG06} W. Gao and A. Geroldinger, {\it Zero-sum problems in finite abelian groups: a survey}, Expo. Math.,
{\bf 24} (2006), 337-369

\bibitem{GGG10} W. Gao, A. Geroldinger, and D. Grynkiewicz, {\it Inverse zero-sum problems III}, Acta Arith.,
{\bf 141} (2010), 103-152.

\bibitem{GS92} A. Geroldinger and R. Schneider, {\it On Davenports constant}, J. Comb. Theory Ser. A, {\bf 61} (1992),
147-152.

	\bibitem{Ge09a}
	A.~Geroldinger, \emph{Additive group theory and non-unique factorizations},
	Combinatorial {N}umber {T}heory and {A}dditive {G}roup {T}heory, Advanced
	Courses in Mathematics CRM Barcelona, Birkh{\"a}user, 2009, pp.~1 -- 86.
	
	\bibitem{Ge-HK06a}
	A.~Geroldinger and F.~Halter-Koch, \emph{Non-{U}nique {F}actorizations.
		{A}lgebraic, {C}ombinatorial and {A}nalytic {T}heory}, Pure and Applied
	Mathematics, vol. 278, Chapman \& Hall/CRC, 2006.

\bibitem{Gi18} B. Girard, \emph{An asymptotically tight bound for the Davenport constant}, J. \' Ec. polytech. Math. \textbf{5}(2018): 605-611.

\bibitem{GS19}
B. Girard and W. Schmid, \emph{Direct zero-sum problems for certain groups of rank three}, Journal of Number Theory \textbf{197}(2019): 297-316.
	
\bibitem{GS20}
B. Girard and W. Schmid, \emph{Inverse zero-sum problems for certain groups of rank three}, Acta Mathematica Hungarica \textbf{160}(2020): 229-247.

\bibitem{[G]} D. Grynkiewicz,  {\it Structural Additive Theory}, Developments in Mathematics, 30, Springer, Cham, 2013.
	
	\bibitem{Ko-Kr10a}
	 M. Kohls and H. Kraft, \emph{Degree bounds for separating
		invariants}, Math. Res. Lett \textbf{17} (2010), no.~06, 1171--1182.

\bibitem{C20} C. Liu, {\it On the lower bounds of Davenport's constant}, J. Comb. Theory Ser. A, {\bf 171} (2020), 105162-105162.

\bibitem{M92} M. Mazur, {\it A note on the growth of Davenport's constant}, Manuscripta Mathematica, {\bf 74.1}
 (1992), 229-235.
 
\bibitem{N16} E. Noether, Der Endlichkeitssatz der Invarianten endlicher Gruppen, 
Math. Ann. 77 (1916), 89-92.

\bibitem{PS11} A. Plagne and  W. Schmid, {\it An application of coding theory to estimating Davenport constants}, Designs, Codes and Crytography, (2011).


\bibitem{CS14} S. Savchev and F. Chen, Long minimal zero-sum sequences in the group $C_2^{r-1} \oplus C_{2k}$, Integers. 14 (2014).

\bibitem{S11} W. Schmid, {\it The inverse problem associated to the Davenport constant for $C_2\oplus C_2\oplus C_{2n}$ and applications to the arithmetical characterization of class groups}, Electron. J. Combin., {\bf 18.1} (2011), P33.
	
	\bibitem{Sc91a}
	B. Schmid, \emph{Finite groups and invariant theory}, in {T}opics in
	{I}nvariant {T}heory, Lecture Notes in Math., vol. 1478, Springer, 1991,
	pp.~35 -- 66.
	
	\bibitem{Sc24a}
	B. Schefler, \emph{The separating noether number of abelian groups of rank
		two}, Journal of Combinatorial Theory, Series A \textbf{209} (2025), 105951.
	
	\bibitem{Sc25a}
	B. Schefler, \emph{The separating noether number of the direct sum of
		several copies of a cyclic group}, Proceedings of the American Mathematical
	Society \textbf{153} (2025), no.~01, 69--79.
	
\bibitem{Z23} K. Zhao, On Davenport constant of the group $C_2^{r-1} \oplus C_{2k}$, Electron. J. Combin. \textbf{30}.1 (2023), P1.498.

\end{thebibliography}
\end{document}